\documentclass[12pt]{article}
\usepackage[T1]{fontenc}
\usepackage{lmodern,amsmath,amsthm,amsfonts,amssymb,graphicx,float,microtype,thmtools,underscore,mathtools,xurl,multirow}
\usepackage[usenames,dvipsnames,svgnames,table]{xcolor}
\usepackage[shortlabels]{enumitem}
\setlist[itemize]{topsep=0ex,itemsep=0ex,parsep=0ex}
\setlist[enumerate]{topsep=0ex,itemsep=0ex,parsep=0ex}
\usepackage{pgf,tikz,ifthen,arrayjobx}
\usetikzlibrary{arrows}
\usetikzlibrary{calc}
\usepackage[unicode=true]{hyperref}
\hypersetup{
colorlinks,
breaklinks=true,
linkcolor={blue!60!black},
citecolor={black},
urlcolor={blue!60!black},
pdftitle={Induced Subgraph Bounds on the Zero Forcing Number and Chromatic Consequences}}
\usepackage[capitalise, compress, nameinlink, noabbrev]{cleveref}
\crefname{lem}{Lemma}{Lemmas}
\crefname{thm}{Theorem}{Theorems}
\crefname{prop}{Proposition}{Propositions}
\crefname{cor}{Corollary}{Corollaries}

\usepackage[longnamesfirst,numbers,sort&compress]{natbib}
\makeatletter
\def\NAT@spacechar{~}
\makeatother
\usepackage[margin=30mm]{geometry}
\renewcommand{\baselinestretch}{1.15}
\allowdisplaybreaks

\renewcommand{\epsilon}{\varepsilon}
\renewcommand{\emptyset}{\varnothing}
\renewcommand{\ge}{\geqslant}
\renewcommand{\le}{\leqslant}
\renewcommand{\geq}{\geqslant}
\renewcommand{\leq}{\leqslant}

\DeclareMathOperator{\tw}{tw}

\renewcommand{\thefootnote}{\fnsymbol{footnote}}
\theoremstyle{plain}
\newtheorem{thm}{Theorem}

\newtheorem{cor}[thm]{Corollary}
\newtheorem{prop}[thm]{Proposition}

\crefname{obs}{Observation}{Observations}
\newtheorem*{lem*}{Lemma}
\theoremstyle{definition}

\newtheorem*{conj*}{Conjecture}
\newtheorem{remark}{Remark}

\theoremstyle{problem}

\usepackage{thmtools}
\usepackage{thm-restate}

\begin{document}
\title{\bf\boldmath\fontsize{18pt}{18pt}\selectfont
Induced Subgraph Bounds on the Zero Forcing Number
and Chromatic Consequences}

\author{%
Dickson Y. B. Annor\,\footnotemark[5] \qquad
Ben Howerton \,\footnotemark[2] \qquad
}

\date{}

\maketitle

\begin{abstract}
Let $G$ be a graph with chromatic number $\chi(G)$, clique number $\omega(G)$ and zero forcing number $Z(G)$. We establish  new lower bounds on $Z(G)$ in terms of induced triangle-free subgraphs. In particular, we show that if a graph $G$ contains an induced triangle-free subgraph $H$ with minimum degree $\delta(H) \ge 3$, then $Z(G)\ge\delta(H)+1$.
As consequences, we prove that every triangle-free graph
satisfies $\chi(G)\le\max\{3,Z(G)\}$ and obtain an application to
planar graphs. Moreover, we prove that
$\chi(G)\le \frac{Z(G)}{2}+2$ for every triangle-free graph. 


\end{abstract}

\textbf{Keywords:} Chromatic number, clique number, zero forcing number, induced subgraphs, triangle-free graphs.

\textbf{2020 Mathematics Subject Classification:} 05C15, 05C50, 05C69.

\footnotetext[5]{Department of Mathematical and Physical Sciences, La Trobe University, Bendigo, Australia (\texttt{\ d.annor@latrobe.edu.au}). 
}
\footnotetext[2]{\texttt{\ bhowerton@pointdynamics.com}.
}

\renewcommand{\thefootnote}{\arabic{footnote}}


\section{Introduction}
In this paper, we deal with finite undirected and simple graphs.

The chromatic number of a graph is one of the most fundamental parameters in graph theory. Determining the chromatic number of a graph is NP-hard, and consequently a large body of research has focused on bounding it in terms of other graph invariants. Classical examples include Brooks' theorem \cite{brooks1941colouring}, which bounds the chromatic number by the maximum degree, and the trivial lower bound given by the clique number.

The zero forcing number, introduced by the AIM Minimum Rank Special Graphs Work Group \cite{aim2008zero}, is a graph parameter arising from a colour-propagation process. Originally developed to study the minimum rank problem for graphs, zero forcing has since found applications in linear algebra, control theory, network monitoring, and graph searching.

Let $G$ be a graph with all vertices initially coloured either black or white. If $u$ is a black vertex of $G$ and $u$ has exactly one neighbour that is white, say $v$, then we
change the colour of $v$ to black. This rule is called the \emph{colour-change rule}. In this case
we say that $u$ forces $v$ and  write $u \to v$. The \emph{zero forcing number} of $G$, $Z(G)$, is the minimum size of an initial set of black vertices that can force the remaining white vertices of $G$ to become black by repeatedly applying the colour-change rule. 

Since both graph colouring and zero forcing involve propagation constraints on vertex colourings, it is natural to ask whether the two parameters are related. In a notable result, Taklimi \cite{taklimi2013zero} proved the following, which provides a surprisingly direct connection between a colouring parameter and a forcing parameter. 

\begin{thm}[\cite{taklimi2013zero}]\label{thm:takl}
 Any graph $G$ satisfies  
 $\chi(G) \leq Z(G)+1$.
\end{thm}



The main contribution of this paper is the establishment of new lower bounds for the zero forcing number that arise from induced subgraphs. Taklimi's proof of $\chi(G)\le Z(G)+1$ is based on the observation that $Z(G)$ is at least the degeneracy of $G$. We show that, under additional structural assumptions, one can improve this lower bound by one unit.

Our main results are the following.

\begin{thm}\label{thm:induceddk}
Let $G$ be a graph, and let $H$ be an induced triangle-free subgraph of $G$ with
$\delta(H)\ge 3$. Then
\[
Z(G)\ \ge\ \delta(H)+1 .
\]
\end{thm}

Theorem~\ref{thm:induceddk} can be viewed as an induced-subgraph version of the Davila--Kenter \cite{davila2014bounds} bound for triangle-free graphs and is best possible. As a consequence, 
we prove the following. 

\begin{cor}\label{cor:tfchi}
If $G$ is a triangle-free graph, then $\chi(G)\le\max\{3,\,Z(G)\}$. In particular every
triangle-free graph $G$ with $Z(G) \ge3$ satisfies $\chi(G)\le Z(G)$, improving
Theorem~\ref{thm:takl} by one in this range.
\end{cor}

Moreover, we prove that the following stronger inequality
also holds 
for
triangle-free graphs.

\begin{thm}\label{thm:freezfn}
Every triangle-free graph $G$ satisfies
\[
\chi(G) \leq \frac{Z(G)}{2}+2.
\]
\end{thm}

The family of odd cycles shows that the bound in Theorem~\ref{thm:freezfn} is sharp. Moreover, Theorem~\ref{thm:freezfn} is at least as strong as Theorem~\ref{thm:takl} for every nonempty triangle-free graph, and it is frequently substantially stronger within this family. 





An earlier version of this paper proposed the inequality
\begin{equation}\label{eq:former-bound}
\chi(G)\le
\left\lceil\frac{\omega(G)+Z(G)+1}{2}\right\rceil .
\end{equation}
After that version appeared, a counterexample on $13$ vertices was
communicated to us. This prompted a targeted exact search, which found the
$12$-vertex counterexample described in Subsection~\ref{sec:verification}.
Together with our exhaustive census through $11$ vertices, it establishes
that $12$ is the minimum possible order of a counterexample.

The induced-subgraph theorems and their triangle-free consequences are
independent of \eqref{eq:former-bound}. We therefore refocus the paper on
those results. We retain the exhaustive computation through eleven vertices
as an exact finite result, not as evidence for a universal conjecture.










The following results are well-known.
\begin{thm}[\cite{diestel2024graph}]\label{thm:chiclique}
Any graph $G$ satisfies $\chi(G) \geq \omega(G)$.
\end{thm}

\begin{thm}[\cite{brooks1941colouring}]\label{thm:brooks}
If $G$ is a connected graph with maximum degree $\Delta(G)$,
then $\chi(G)\le\Delta(G)$ unless $G$ is a complete graph or an odd cycle.
\end{thm}

From Theorems~\ref{thm:takl} and~\ref{thm:brooks}, we can easily deduce the following.

\begin{prop}\label{prop:averagebdd}
 For any graph $G$ with maximum degree $\Delta(G)$ and zero forcing number $Z(G)$, it holds that
 \begin{equation*}
\chi(G) \leq  \left \lceil \frac{\Delta(G)+Z(G)+1}{2}\right\rceil
 \end{equation*}
and this bound is tight.
\end{prop}










The following results will be useful in our work.



\begin{thm}[\cite{gentner2016extremal}]\label{thm:tri-free}
Let $G$ be a triangle-free graph with minimum degree $\delta \geq 2$. Then $Z(G) \geq 2\delta(G) -2$.
\end{thm}

\begin{thm}[\cite{grotzch1958theorie}]\label{thm:3col}
 Every triangle-free planar graph can be coloured with three colours.
\end{thm}

The paper is organised as follows. Section~\ref{sec:induced} develops new
induced lower bounds for the zero forcing number and derives several
consequences. Section~\ref{sec:omegachiz} records chromatic applications for
triangle-free and planar graphs, reports the exhaustive computation through
eleven vertices, and gives a minimum-order counterexample to
\eqref{eq:former-bound}.

\section{Induced lower bounds for the zero forcing number}\label{sec:induced}


In this section, we establish new lower bounds for the zero forcing number arising from induced subgraphs.

Taklimi's proof of Theorem~\ref{thm:takl} goes through the degeneracy: she shows that
\begin{equation}\label{eq:taklimi-deg}
Z(G)\ \ge\ \max\{\,\delta(H) : H \text{ an induced subgraph of } G\,\}
\end{equation}
(the right-hand side is the degeneracy $d(G)$), and combines it with the greedy
bound $\chi\le d(G)+1$ \cite{taklimi2013zero}.

We show that the presence of an induced triangle-free subgraph with minimum degree at least three forces a stronger bound. 




\begin{proof}[Proof of Theorem~\ref{thm:induceddk}]

Let $k=\delta(H)\ge3$, and suppose for a contradiction that $Z:=Z(G)=k$ (by
\eqref{eq:taklimi-deg}, $Z\ge k$). Fix a minimum zero forcing set $B$, $|B|=k$,
and a chronological list of forces, serialised so that exactly one force occurs
at each step $1,2,\dots$; for a vertex $z$ let $t(z)$ be the step at which $z$
becomes black ($t(z)=0$ for $z\in B$). The forces partition $V(G)$ into $Z$
forcing chains \cite{barioli2010zero}: induced paths $v_1\to v_2\to\cdots$,
where $v_1\in B$ and $v_i$ forces $v_{i+1}$. Times increase strictly along each
chain; each vertex forces at most once, and a vertex that has forced has no white
neighbours afterwards. Note $|V(H)|\ge k+1$.

\emph{Step 0: some vertex of $H$ performs a force.} Otherwise every vertex of
$H$ is the final vertex of its chain, so $|V(H)|\le Z=k<k+1$, a contradiction.

Let $v$ be the chronologically first vertex of $H$ to perform a force, at step
$t^\ast$, forcing a vertex $x$.

\emph{Step 1: before step $t^\ast$, each chain carries at most one black vertex
of $H$.} If $a,b\in V(H)$ lie on one chain with $t(a)<t(b)<t^\ast$, then $a$ is
not the final vertex of its chain (b follows it), so $a$ forces its
chain-successor $s$ with $t(s)\le t(b)<t^\ast$; so $a\in V(H)$ performs a force
before $v$ does, a contradiction. Moreover no vertex of $H$ other than $v$ on
$v$'s chain is black before $t^\ast$: a vertex before $v$ would force before
step $t^\ast$ as above, and every vertex after $v$ has time $\ge t^\ast$.

\emph{Step 2: $x\in V(H)$, $d_H(v)=k$, and the set of vertices of $H$ black
before step $t^\ast$ is exactly $N_H[v]\setminus\{x\}$, with precisely one on
each of the $k$ chains.}
At step $t^\ast$, $v$ forces $x$, so every neighbour of $v$ other than $x$
is black before $t^\ast$. If $x\notin V(H)$, then $\{v\}\cup N_H(v)$ consists
of $1+d_H(v)\ge k+1$ black vertices of $H$ lying (Step 1) on pairwise distinct
chains, so $Z\ge k+1$, a contradiction; hence $x\in V(H)$. Now
$\{v\}\cup\bigl(N_H(v)\setminus\{x\}\bigr)$ consists of $d_H(v)$ black vertices
of $H$ on distinct chains, so $d_H(v)\le Z=k$, whence $d_H(v)=k$; and any
further black vertex of $H$ would occupy a $(k+1)$-st chain. So the black
vertices of $H$ are $\{v,y_1,\dots,y_{k-1}\}$ where
$N_H(v)=\{x,y_1,\dots,y_{k-1}\}$, one on each chain.

\emph{Step 3: white reservoirs.} Since $H$ is triangle-free and
$x,y_1,\dots,y_{k-1}$ are common neighbours of $v$, the set
$\{x,y_1,\dots,y_{k-1}\}$ is independent in $H$. Let
$W=V(H)\setminus N_H[v]$; every vertex of $W$ is white at step $t^\ast$. Each
$y_i$ has $d_H(y_i)\ge k$ neighbours in $H$; removing $v$ and using
independence, $y_i$ has at least $k-1\ge2$ neighbours in $W$. Likewise $x$ has
at least $k-1\ge2$ neighbours in $W$. In particular $W\neq\emptyset$.

\emph{Step 4: no vertex of $H$ forces until a vertex of $W$ is black.} Since $B$
is a zero forcing set, every vertex of $W$ becomes black eventually; let $w$ be
the first, at step $s>t^\ast$, forced by some $u$ ($w\notin B$ since $w$ is white
at $t^\ast$). The vertices of $H$ black before step $s$ are exactly
$\{v,y_1,\dots,y_{k-1},x\}$. Suppose some $h\in V(H)$ performs a force at a step
$\sigma\in(t^\ast,s]$; then $h\in\{v,y_1,\dots,y_{k-1},x\}$. Not $h=v$: $v$ has
already forced. If $h=y_i$ (or $h=x$): when $h$ forces, all neighbours of $h$
except the forced vertex are black; but $h$ has at least $k-1\ge 2$ neighbours
in $W$, so at least $k-2\ge1$ of them, distinct from the forced vertex, would
have to be black before step $s$ --- impossible, as no vertex of $W$ is black
before step $s$. So no vertex of $H$ forces at any step in $(t^\ast,s]$; in
particular $u\notin V(H)$.

\emph{Step 5: contradiction.} Consider the chain $C$ containing $w$. By Step 2,
$C$ carries exactly one vertex $\tau\in\{v,y_1,\dots,y_{k-1}\}$ black before
$t^\ast$, and $\tau$ is the first vertex of $H$ on $C$ (an earlier one would be
black before $t(\tau)\le t^\ast$, contradicting Step 2). Since
$t(w)=s>t^\ast\ge t(\tau)$, $w$ lies after $\tau$ on $C$.

First suppose $C$ is $v$'s chain, so $\tau=v$ and the successor of $v$ on $C$ is
$x$, and $w$ lies after $x$. Let $\sigma_1$ be the vertex immediately after $x$
on $C$ (possibly $\sigma_1=w$). Then $x$ forces $\sigma_1$ at step
$t(\sigma_1)\le t(w)=s$, and $t(\sigma_1)>t^\ast$ since at step $t^\ast$ the
unique force is $v\to x$. So $x\in V(H)$ performs a force in $(t^\ast,s]$,
contradicting Step 4.

Otherwise $\tau=y_i$ for some $i$. Let $\sigma_1$ be the vertex immediately
after $\tau$ on $C$ (it exists, as $w$ lies after $\tau$; possibly
$\sigma_1=w$). Then $y_i$ forces $\sigma_1$ at step $t(\sigma_1)\le s$; and
$t(\sigma_1)>t^\ast$, for otherwise $y_i\in V(H)$ performed a force before $v$
(or $\sigma_1=x$, which lies on $v$'s chain, not $C$).
So $y_i$ performs a force in $(t^\ast,s]$, contradicting Step 4.

In either case we have a contradiction, so $Z(G)\ge k+1$. This completes the proof.
\end{proof}

\begin{remark}\label{rem:sharp}
Theorem~\ref{thm:induceddk} is sharp in all respects.
\begin{enumerate}
    \item The conclusion cannot be improved to $Z(G)\ge\delta(H)+2$: $K_{3,3}$ (as
$H=G$) has $\delta=3$ and $Z=4$. Thus, in contrast with the bound
$Z(G)\ge2\delta(G)-2$ of \cite{gentner2016extremal}, only one unit survives the
    passage to induced subgraphs.
\item Triangle-freeness cannot be dropped since the octahedron $K_{2,2,2}$ has
$\delta=4$, $\omega=3$ (so no $K_\delta$, even) and $Z=4$.
\item $\delta(H)\ge3$ is needed as cycles have $\delta=2=Z$.
\end{enumerate}

\end{remark}


\begin{proof}[Proof of Corollary~\ref{cor:tfchi}]
Let $d$ be the degeneracy of $G$, and $H$ an induced subgraph with
$\delta(H)=d$; $H$ is triangle-free. If $d\le2$, then $\chi(G)\le d+1\le3$. If
$d\ge3$, Theorem~\ref{thm:induceddk} gives $Z(G)\ge d+1$, so
$\chi(G)\le d+1\le Z(G)$.
\end{proof}



\begin{thm}\label{thm:master}
Let $H$ be an induced subgraph of $G$ with $\delta(H)=k\ge2$ such that, for
every vertex $v$ of $H$ with $d_H(v)=k$, every $u\in N_H(v)$ has at least two
neighbours in $V(H)\setminus N_H[v]$. Then $Z(G)\ge k+1$.
\end{thm}

\begin{proof}
This is identical to the proof of Theorem~\ref{thm:induceddk}. Steps 0--2 force the
first $H$-forcer $v$ to have $d_H(v)=k$, so the hypothesis applies to it and
keeps every potential forcer in $N_H(v)\cup\{x\}$ blocked (two white
neighbours outside $N_H[v]$) until the contradiction of Step 5.
\end{proof}

\begin{thm}\label{thm:conedk}
Let $H'$ be an induced triangle-free subgraph of $G$ with $\delta(H')=k\ge3$,
and let $A$ be a clique of $G-V(H')$, possibly empty, each vertex of which is
adjacent to every vertex of $H'$. Then
$Z(G)\ \ge\ |A|+k+1\ =\ \delta(A\vee H')+1$.
\end{thm}

\begin{proof}
Put $H=A\vee H'$ and $a=|A|$. Since $H'$ is triangle-free and has minimum
degree $k$, choosing an edge $uv$ of $H'$ gives
$N_{H'}(u)\cap N_{H'}(v)=\varnothing$, and hence
$|V(H')|\ge d_{H'}(u)+d_{H'}(v)\ge2k$. Therefore every vertex of $A$ has
degree
\[
|V(H')|+a-1\ge2k+a-1>a+k
\]
in $H$, whereas a vertex of degree $k$ in $H'$ has degree $a+k$ in $H$.
Thus $\delta(H)=a+k$, and every minimum-degree vertex $v$ of $H$ lies in
$H'$ and satisfies $d_{H'}(v)=k$.

Fix such a vertex $v$ and put
$W=V(H')\setminus N_{H'}[v]$. We have
$|W|\ge2k-(k+1)=k-1\ge2$. If $u\in N_{H'}(v)$, triangle-freeness implies
that the other $d_{H'}(u)-1\ge k-1\ge2$ neighbours of $u$ in $H'$ all lie
in $W$. If $u\in A$, then $u$ is adjacent to every vertex of $W$. Hence
every $u\in N_H(v)$ has at least two neighbours in
$V(H)\setminus N_H[v]=W$. Theorem~\ref{thm:master} now gives
$Z(G)\ge\delta(H)+1=a+k+1$.
\end{proof}

\section{Chromatic consequences}\label{sec:omegachiz}

Dvo\v{r}\'ak and
Kawarabayashi~\cite{dvorak2017trianglefree} proved that every triangle-free
graph of treewidth at most $t$ has chromatic number at most
$\lceil(t+3)/2\rceil$, while Barioli et al.~\cite{barioli2013parameters}
proved that $\tw(G)\le Z(G)$. Using these two results, we prove Theorem~\ref{thm:freezfn}. 

\begin{proof}[Proof of Theorem~\ref{thm:freezfn}]
Let $t=\tw(G)$. If $t=0$, then $G$ is edgeless and the result is immediate.
For $t\ge1$, the theorem of Dvo\v{r}\'ak and
Kawarabayashi~\cite{dvorak2017trianglefree}, together with
$\tw(G)\le Z(G)$~\cite{barioli2013parameters}, gives
\[
\chi(G)
 \le \left\lceil\frac{\tw(G)+3}{2}\right\rceil
 \le \left\lceil\frac{Z(G)+3}{2}\right\rceil
 = \left\lfloor\frac{Z(G)}{2}\right\rfloor+2
 \le \frac{Z(G)}{2}+2.
\]
\end{proof}

The computation in Subsection~\ref{sec:verification} supplies an independent
check of Theorem~\ref{thm:freezfn} through order $12$. For completeness, the
following elementary argument proves the near-regular case without using
treewidth.

\begin{thm}\label{thm:allregular}
If $G$ is a triangle-free graph  with
$\Delta(G)\le\delta(G)+1$, then
$\displaystyle\chi(G)\le \frac{Z(G)}2+2$.
\end{thm}

\begin{proof}
First suppose $\delta(G)\le1$. Then $\Delta(G)\le2$. If $\chi(G)\le2$, the
result is immediate since every nonempty graph has $Z(G)\ge1$. If
$\chi(G)=3$, some component of $G$ is an odd cycle, so $Z(G)\ge2$ and
$\chi(G)=3\le Z(G)/2+2$.

Now suppose $\delta(G)\ge2$. By Theorem~\ref{thm:tri-free},
$Z(G)\ge2\delta(G)-2$, and hence
\[
\frac{Z(G)}2+2\ge\delta(G)+1\ge\Delta(G).
\]
If $\chi(G)\le\Delta(G)$, this proves the result. Otherwise a component
attaining $\chi(G)$ is, by Brooks' theorem, a complete graph or an odd
cycle. The complete-graph alternative is impossible for a triangle-free
graph with minimum degree at least two. In the odd-cycle alternative,
$\chi(G)=3$ and $Z(G)\ge2$, so again
$\chi(G)\le Z(G)/2+2$.
\end{proof}

An immediate consequence of Theorem~\ref{thm:allregular} is the following. 

\begin{cor}\label{cor:regular}
Every triangle-free $r$-regular graph $G$ satisfies
$\displaystyle\chi(G)\le \frac{Z(G)}2+2$.
\end{cor}

We next prove that every planar graph satisfies the former proposed
inequality.

\begin{thm}\label{thm:omegazfnplanar}
Every planar graph $G$ satisfies
\[
\chi(G)\le
\left\lceil\frac{\omega(G)+Z(G)+1}{2}\right\rceil.
\]
\end{thm}
\begin{proof}
By the Four Colour Theorem, $\chi(G)\le4$ \cite{appel1976every}. If
$\chi(G)\le1$, the result is immediate. If $\chi(G)=2$, then $G$ has an
edge, so $\omega(G)\ge2$, while Theorem~\ref{thm:takl} gives $Z(G)\ge1$.
If $\chi(G)=3$, then again $\omega(G)\ge2$, and
Theorem~\ref{thm:takl} gives $Z(G)\ge2$. These two cases give respectively
\[
2\le\left\lceil\frac{2+1+1}{2}\right\rceil
\quad\text{and}\quad
3\le\left\lceil\frac{2+2+1}{2}\right\rceil.
\]
Finally, if $\chi(G)=4$, then $G$ is not triangle-free by
Theorem~\ref{thm:3col}; hence $\omega(G)\ge3$, and
Theorem~\ref{thm:takl} gives $Z(G)\ge3$. Therefore
\[
4=\left\lceil\frac{3+3+1}{2}\right\rceil
\le\left\lceil\frac{\omega(G)+Z(G)+1}{2}\right\rceil.
\]
\end{proof}

\subsection{Small-order computation and a minimum counterexample}
\label{sec:verification}

Using \texttt{geng} (nauty) \cite{mckay2014practical} to enumerate connected graphs up to isomorphism and
computing $\chi(G)$, $\omega(G)$ and $Z(G)$ exactly for a graph $G$ ($Z(G)$ by exhausting initial sets with
closure evaluation; the engine reproduces the known values for $C_5$, $C_7$,
$K_5$, $K_{3,3}$, the $5$-wheel, the Petersen graph, and the Gr\"otzsch graph),
we obtained the following exact finite results. In the
first version these computations were presented as evidence for
\eqref{eq:former-bound}; here we retain them as a small-order census and use
them to establish the minimum order of a counterexample. The order-$11$
calculation was rerun in $256$ shards; the source code, per-shard summaries,
and SHA-$256$ checksums are retained with the ancillary computational files.

\begin{itemize}[topsep=2pt,itemsep=2pt]
\item Inequality~\eqref{eq:former-bound} holds for all
$1{,}018{,}690{,}329$ connected graphs with at most $11$ vertices
($1{,}006{,}700{,}565$ of them on exactly $11$ vertices). No violations.
\item As an independent check of Theorem~\ref{thm:freezfn}, the inequality
holds for all $1{,}144{,}061$ connected triangle-free graphs on $12$ vertices.
Exactly $23$ of these graphs are $4$-chromatic: $16$ have $Z=5$ and $7$ have
$Z=6$. Together with the census through $11$ vertices, this computationally
confirms the theorem through order $12$.
\end{itemize}

\begin{prop}\label{prop:mincounterexample}
Inequality~\eqref{eq:former-bound} holds for every graph on at most $11$
vertices and fails for a graph on $12$ vertices. Consequently, the minimum
order of a counterexample is $12$.
\end{prop}

\begin{proof}
The exhaustive census above proves the assertion through order $11$ for
connected graphs. This is sufficient: if a disconnected graph $G$ violated
\eqref{eq:former-bound}, a component $C$ with $\chi(C)=\chi(G)$ would satisfy
$\omega(C)\le\omega(G)$ and $Z(C)\le Z(G)$ and would itself be a
counterexample.

For order $12$, consider the graph with graph6 representation
\verb|K~nXzlFWxDkF|. It has $40$ edges and
\[
(\chi(G),\omega(G),Z(G))=(6,4,5).
\]
For directly checkable lower and upper certificates, the vertices
$\{0,1,2,3\}$ form a $K_4$, and
\[
(4,2,1,3,5,1,2,4,3,5,4,6)
\]
is a proper $6$-colouring of vertices $0,\ldots,11$. Exact backtracking finds
no proper $5$-colouring and exact clique search finds no $K_5$. An
independent computation with nauty confirms $\chi(G)=6$ and
$\omega(G)=4$.

The set $\{0,1,2,3,4\}$ is a zero forcing set, with forcing sequence
\[
0\to5,\quad 4\to6,\quad 3\to7,\quad 5\to8,\quad
7\to9,\quad 1\to11,\quad 2\to10.
\]
Since $\delta(G)=5$, \eqref{eq:taklimi-deg} gives $Z(G)\ge5$, so the displayed
sequence proves $Z(G)=5$. Therefore
\[
2\chi(G)=12>11=\omega(G)+Z(G)+2,
\]
which is equivalent to the failure of \eqref{eq:former-bound}.
\end{proof}

\begin{figure}[htbp]
\centering
\includegraphics[width=.88\linewidth]{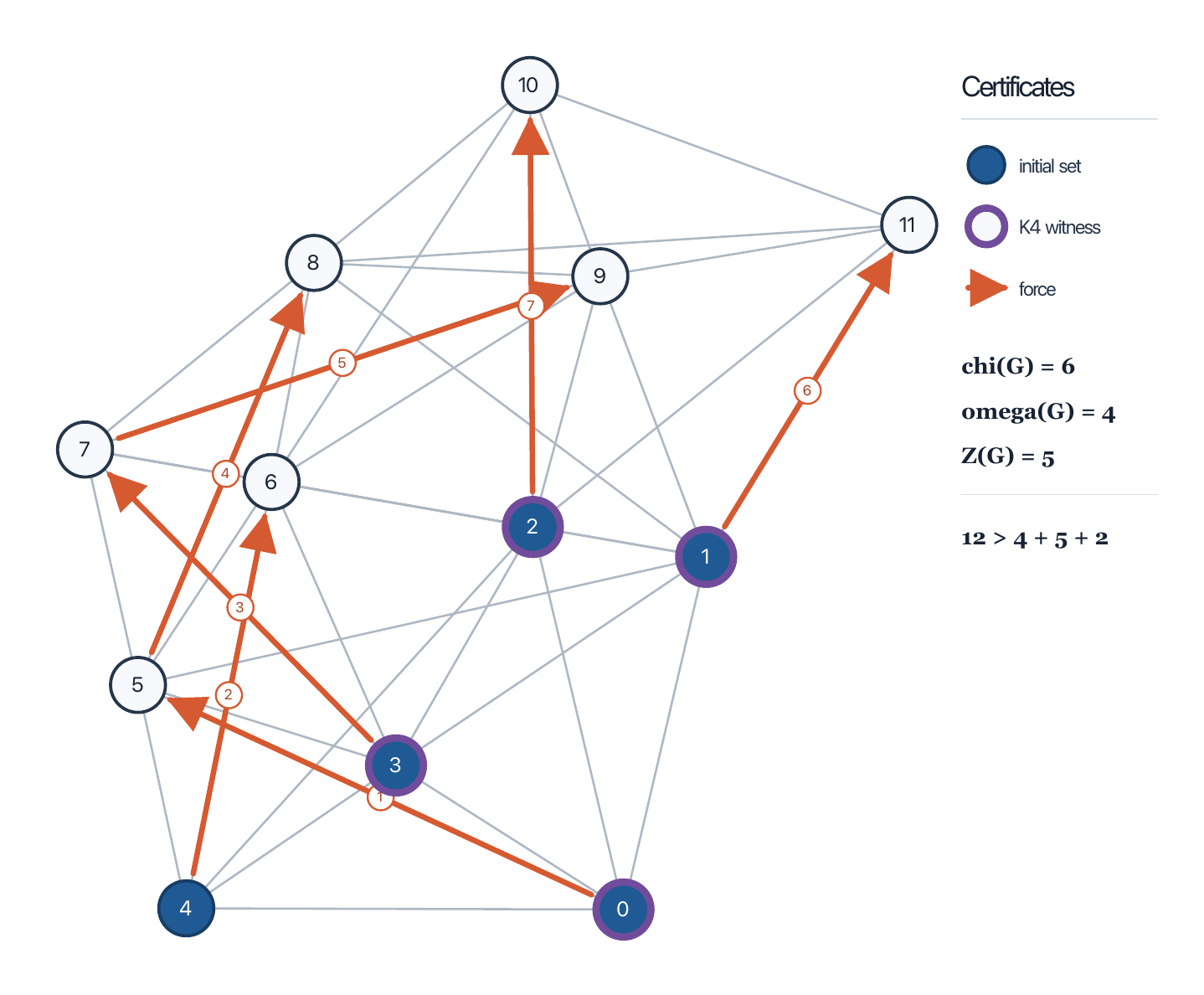}
\caption{A minimum-order counterexample to \eqref{eq:former-bound}. Blue
vertices form an initial zero forcing set, purple double rings mark a
$K_4$, and the numbered arrows give a forcing sequence.}
\label{fig:order12-counterexample}
\end{figure}

\section{Concluding remarks}
The counterexample to \eqref{eq:former-bound} shows that clique number and
zero forcing number alone do not control chromatic number through the
proposed averaging formula. The induced-subgraph results proved here remain
valid and suggest a different direction: identify structural hypotheses on
an induced subgraph $H$ that force $Z(G)$ to exceed $\delta(H)$.
Theorem~\ref{thm:induceddk} resolves this when $H$ is triangle-free with
minimum degree at least three, and Theorem~\ref{thm:master} isolates a more
general local criterion.

Theorem~\ref{thm:freezfn} settles the triangle-free inequality for all
triangle-free graphs by combining known treewidth results, so the
near-regularity condition in Theorem~\ref{thm:allregular} is not needed for
that conclusion. A natural problem remains to determine other induced graph
classes for which the degeneracy lower bound on $Z(G)$ can be improved. This
preserves the connection between zero forcing and graph colouring while
respecting the obstruction exhibited by the minimum-order counterexample in
Proposition~\ref{prop:mincounterexample}. It would also be interesting to
characterise the counterexamples of minimum order.

\paragraph{Acknowledgements.}
We thank Digvijay Bokey for communicating an earlier $13$-vertex counterexample,
which prompted the targeted search reported here.

\bibliographystyle{plain}
\bibliography{references}

\end{document}